\documentclass{amsart}

\usepackage{ae,aecompl}
\usepackage{esint}
\usepackage{graphicx}
\usepackage[all,cmtip]{xy}
\usepackage{amsmath, amscd}
\usepackage{booktabs}
\usepackage{amsthm}
\usepackage{amssymb}
\usepackage{amsfonts}
\usepackage{qsymbols}
\usepackage{latexsym}
\usepackage{mathrsfs}
\usepackage{cite}
\usepackage{color}
\usepackage{url}
\usepackage{enumerate}

\usepackage{verbatim}
\usepackage[draft=false, colorlinks=true]{hyperref}
\usepackage{orcidlink}

\allowdisplaybreaks

\newcommand {\bmo}{\mathrm{bmo}}
\newcommand {\BMO}{{\mathrm{BMO}}}

\newcommand {\C}{{\mathbb C}}

\newcommand {\ud}{\mathrm{d}}

\newcommand {\veps}{\varepsilon}

\newcommand {\HT}{\mathcal{H}}

\newcommand {\ka}{\kappa}

\newcommand {\la}{\lambda}
\newcommand {\rb}{\rangle}
\newcommand {\lb}{{\langle}}

\newcommand {\loc}{{\mathrm{loc}}}

\newcommand {\N}{{{\mathbb N}}}

\newcommand {\ph}{\varphi}

\newcommand {\R}{\mathbb R}

\newcommand {\Rn}{\mathbb{R}^{n}}

\newcommand {\Ric}{\mathrm{Ric}}
\newcommand {\supp}{\mathrm{supp}}

\newcommand {\Sw}{\mathcal{S}}

\newcommand {\Tp}{T^{*}\Rn}

\newcommand\WF{\operatorname{WF}}

\newcommand {\Z}{\mathbb Z}

\newcommand {\vanish}[1]{\relax}

\newcommand{\wh}{\widehat}

\DeclareFontFamily{U}{mathx}{\hyphenchar\font45}
\DeclareFontShape{U}{mathx}{m}{n}{
      <5> <6> <7> <8> <9> <10>
      <10.95> <12> <14.4> <17.28> <20.74> <24.88>
      mathx10
      }{}
\DeclareSymbolFont{mathx}{U}{mathx}{m}{n}
\DeclareFontSubstitution{U}{mathx}{m}{n}
\DeclareMathAccent{\widecheck}{0}{mathx}{"71}

\DeclareMathOperator{\Real}{Re}
\DeclareMathOperator{\Imag}{Im}

\newtheorem{theorem}{Theorem}[section]
\newtheorem{lemma}[theorem]{Lemma}
\newtheorem{proposition}[theorem]{Proposition}
\newtheorem{corollary}[theorem]{Corollary}

\theoremstyle{definition}
\newtheorem{definition}[theorem]{Definition}
\newtheorem{remark}[theorem]{Remark}

\numberwithin{equation}{section}
\protected\def\ignorethis#1\endignorethis{}
\let\endignorethis\relax

\title[Propagation of singularities for $C^{r}$ coefficients with $r>1$]{Propagation of singularities for equations with $C^{r}$ coefficients for $r>1$}

\author{Jan Rozendaal}
\address{Institute of Mathematics, Polish Academy of Sciences\\
ul.~\'{S}niadeckich 8\\
00-656 Warsaw\\
Poland}
\email{jrozendaal@impan.pl}



\thanks{This research was funded in part by the National Science Center, Poland, grant 2021/43/D/ST1/00667. The author is partially supported by NCN grant UMO-2023/49/B/ST1/01961.}

\begin{document}

\begin{abstract}
We observe that, for $r>1$, $s$ in an $r$-dependent interval, $p$ a homogeneous pseudodifferential symbol of order $m$ having $C^{r}$ regularity in space, and $u\in H^{s+m-r}(\Rn)$ such that $p(x,D)u\in H^{s}(\Rn)$, each point in the $H^{s+m-1}$ wavefront set of $u$ lies on a maximally extended null bicharacteristic of $p$ which is contained in the $H^{s+m-1}$ wavefront set of $u$. In fact, for $r=2$ slightly less than $C^{1,1}$ regularity suffices, and here the results apply to manifolds with bounded Ricci curvature. 
\end{abstract}
	
\maketitle

\section{Introduction}\label{sec:intro}

In this article we consider a version for linear equations with rough coefficients of the classical statement on propagation of singularities.

\subsection{Main results}

Our results are formulated using pseudodifferential symbols that have limited regularity in the spatial variable. 

Loosely speaking, for $r>0$ and $m\in\R$, the class $C^{r}_{*}S^{m}_{1,0}$ from Definition \ref{def:symbolrough} consists of symbols 
that behave like an element of $S^{m}_{1,0}$ in the fiber variable, but are elements of the Zygmund space $C^{r}_{*}(\Rn)$ from Definition \ref{def:Zygmund} in the spatial variable. For $r\notin \N$ the latter space coincides with $C^{r}(\Rn)$, whereas for $r\in\N$ it is a strict subset of $C^{r-1,1}(\Rn)\subseteq C^{r}(\Rn)$. 
For such a symbol $p$ and for 
$(x,\xi)\in T^{*}\Rn\setminus o=\R^{2n}\setminus (\Rn\times\{0\})$, we write 
\begin{equation}\label{eq:phat}
\hat{p}(x,\xi):=\lim_{\la\to\infty}\la^{-m}p(x,\la\xi).
\end{equation}
Moreover, $\Sigma_{\hat{p}}=\hat{p}^{-1}(\{0\})$ is the characteristic set of $\hat{p}$, cf.~\eqref{eq:Sigma}. 

We refer to Section \ref{subsec:microlocal} for other basic notions from microlocal analysis which appear in the following theorem, the main result of this article.

\begin{theorem}\label{thm:mainintro}
Let $r>1$, $m\in\R$ and $-r<s<r$. Let $p\in C^{r}_{*}S^{m}_{1,0}$ be such that $\Imag p\in C^{r}_{*}S^{m-1}_{1,0}$, and such that $\hat{p}$ is a well-defined element of $C^{1}(\Tp\setminus o)$. Set
\begin{equation}\label{eq:sigma}
\sigma:=\begin{cases}
s-r&\text{if }0<s<r,\\
\veps-r&\text{if }-r<s\leq 0,
\end{cases}
\end{equation}
for $\veps>0$. Let $u\in H^{\sigma+m}(\Rn)$ be such that $p(x,D)u\in H^{s}(\Rn)$. Then $\WF^{s+m}u\subseteq\Sigma_{\hat{p}}$. Moreover, for each $(x,\xi)\in\WF^{s+m-1}u$ there exists a null bicharacteristic $\gamma:\R\to\Tp\setminus o$ of $\hat{p}$ satisfying $\gamma(0)=(x,\xi)$ and $\gamma(t)\in \WF^{s+m-1}u$ for all $t\in\R$. 
\end{theorem}

Theorem \ref{thm:mainintro} follows from a microlocal statement in Theorem \ref{thm:main}. The parameter $\sigma$ in \eqref{eq:sigma} measures the a priori regularity of $u$ needed to leverage $H^{s}(\Rn)$ regularity of $p(x,D)u$.

Under the assumptions of Theorem \ref{thm:main}, the Hamilton vector field $H_{\hat{p}}$ is continuous, and bounded on the cosphere bundle $\Rn\times S^{n-1}$. Moreover, $\hat{p}$ is positively homogeneous of degree $m$ in the fiber variable. Hence, by Peano's theorem (see e.g.~\cite[Proposition 1.A.1]{Taylor23a}), integral curves $\gamma:\R\to \Tp\setminus o$ of $H_{\hat{p}}$ with specified initial data exist, and every point in the wavefront set lies on a null bicharacteristic which is completely contained in the wavefront set. As such, one might say that the wavefront set is a union of maximally extended null bicharacteristics. However, if $\hat{p}$ is merely continuously differentiable, then integral curves of $H_{\hat{p}}$ need not be unique, and Theorem \ref{thm:main} does not preclude certain bicharacteristics from entering the wavefront set and then leaving it again. On the other hand, if $\hat{p}\in C^{1,1}(\Tp\setminus o)$, then bicharacteristics are unique and one recovers the classical statement on propagation of singularities, albeit under suitable a priori regularity assumptions on $u$. In fact, $C^{2}_{*}$ regularity suffices for uniqueness (see Remark \ref{rem:C11}). 

Of course, 
Theorem \ref{thm:mainintro} implies a statement on propagation of regularity. Namely, if $(x,\xi)\in\Tp\setminus o$ is such that every null bicharacteristic $\gamma:\R\to \Tp\setminus o$ of $\hat{p}$ with $\gamma(0)=(x,\xi)$ has non-empty intersection with the complement of $\WF^{s+m-1}u$, then $(x,\xi)\notin \WF^{s+m-1}u$ (see Corollary \ref{cor:propreg1}). Moreover, this statement can be quantified, cf.~Corollary \ref{cor:propreg2}.

Differential operators with coefficients of $C^{r}_{*}(\Rn)$ regularity form a concrete class to which Theorem \ref{thm:mainintro} applies. In fact, such operators are straightforward to define on domains, and a version of Theorem \ref{thm:mainintro} holds for differential operators on domains, Corollary \ref{cor:domains}. Here only local regularity is required of $p$ and $u$. We also consider second-order differential operators in divergence form, in Theorem \ref{thm:divform}. Here the Sobolev interval for $s$ is shifted. 

We note that, if one considers in all these results differential operators that have slightly more spatial regularity, measured using Sobolev spaces $\HT^{r,\infty}(\Rn)$ over the space $\bmo(\Rn)$ from Definition \ref{def:bmo}, then the endpoints $s=-r$ and $s=r$ of the Sobolev interval for $s$ can be included, and one may let $\veps=0$ in \eqref{eq:sigma}. For integer $r$, this applies in particular to operators with $C^{r-1,1}$ regularity (see \eqref{eq:Zygmundemb2}).

Note that, for $r=2$, our results require slightly less than $C^{1,1}$ regularity. A concrete geometric setting in which this is relevant arises when considering wave equations on manifolds with a metric which has bounded Ricci curvature tensor. 
Such a metric automatically has $\HT^{2,\infty}\subseteq C^{2}_{*}$ regularity, and therefore Theorem \ref{thm:divform} yields propagation of singularities for these metrics, as is explained in Section \ref{subsec:Ricci}.

\subsection{Previous work}\label{subsec:previous}

In the smooth setting, where $p\in S^{m}_{1,0}=\cap_{r>0}C^{r}_{*}S^{m}_{1,0}$ and $\hat{p}\in C^{\infty}(\Tp\setminus o)$, Theorem \ref{thm:mainintro} is essentially due to H\"{o}rmander (see \cite{Hormander71a,Hormander71b,Duistermaat-Hormander72}). In this case no a priori regularity is required, and the conclusion holds for general $u\in\Sw'(\Rn)$ and $s\in\R$. 

In \cite{Bony81}, Bony studied propagation of singularities for nonlinear equations with rough solutions, and in the process developed and employed paradifferential calculus in a manner which has played a key role in subsequent work. Treatises of Bony's results involving slight variations on the methods can be found in \cite{Hormander97} and \cite{Taylor91}. In particular, Taylor in \cite{Taylor91} highlights the role of $C^{r}$ regularity of the coefficients. 

In the setting of linear equations with rough coefficients, Taylor in \cite{Taylor00} 
proved Theorem \ref{thm:mainintro} for $r=2$, but with $\sigma$ replaced by $\sigma+\delta$ for any $\delta>0$. 
Results for $1<r<2$ can also be found in \cite{Taylor00}, but these are of a less concrete nature than Theorem \ref{thm:mainintro}. For second-order hyperbolic equations and $r=2$, Smith in \cite{Smith14} obtained the endpoint case $\delta=0$, albeit under the assumption of $C^{1,1}$ regularity. In fact, in the setting of \cite{Smith14} there is a distinguished time variable, and less regularity is required in this distinguished variable than in the remaining ones (see Section \ref{subsec:standwave}). Moreover, \cite{Smith14} contains an example showing that the conclusion of Theorem \ref{thm:mainintro} fails for $r=2$ if $\sigma$ is replaced by $\sigma-\delta$ for some $\delta>0$, at least under the regularity assumptions on the coefficients considered there.

There are various related results on propagation of singularities. 
For example, one also encounters non-uniqueness of bicharacteristics on manifolds with boundary or corners. Note that any smooth Riemannian manifold with boundary can locally be embedded in a manifold without boundary, by reflecting across the boundary. However, the metric that arises in this manner is merely Lipschitz, and as such is too rough to fall within the scope of Theorem \ref{thm:mainintro}. On the other hand, on manifolds with boundary or corners the metric carries additional structure, as does the natural wavefront set of a solution to the wave equation, and it was shown in \cite{Melrose-Sjostrand78,Melrose-Sjostrand82,Vasy08} that this wavefront set is a union of appropriately defined bicharacteristics. Theorem \ref{thm:mainintro} does not require such additional structure but its conclusion is more limited, guaranteeing merely the existence of some maximally extended bicharacteristic which is completely contained in the wavefront set.

Another instance of propagation of singularities for metrics $g$ with singularities of class $C^{r}$ can be found in \cite{deHoUhVa15}, for $1<r<2$ but under the additional geometric assumption that the singularity is conormal. 
It is then shown that, for $u\in H^{s-\rho+2}$ satisfying $(\partial_{t}^{2}-\Delta_{g})u\in H^{s}$, the natural $H^{s+1}$ wavefront set of $u$ is a union of appropriate bicharacteristics. Here $s$ ranges over a nontrivial interval and, as in Theorem \ref{thm:mainintro}, $\rho$ depends on $s$. Nonetheless, the largest $\rho$ can be is slightly less than $(r+1)/2$. Note that this is less than the maximal gain over the background regularity in Theorem \ref{thm:mainintro} for any fixed $r>1$. Hence Theorem \ref{thm:mainintro}, or more precisely Theorem \ref{thm:divform}, does provide some information beyond what is contained in \cite{deHoUhVa15}. 
On the other hand, both the wavefront set and the bicharacteristics in \cite{deHoUhVa15} are tailored to the structure of the problem under consideration, making a strict comparison to Theorem \ref{thm:mainintro} somewhat artificial.

Finally, we note that, again in the setting of the wave equation on a manifold with a rough metric, a propagation result for microlocal defect measures was obtained in \cite{BuDeLe24}. Here the metric is merely assumed to be $C^{1}$, and it is shown that each point in the support of a suitable microlocal defect measure lies on a maximally extended bicharacteristic which is contained in the support of the measure. Such a result is of a slightly different nature than Theorem \ref{thm:mainintro}, but it is worth emphasizing that the mere existence of such bicharacteristics  
is sufficient to obtain concrete results regarding observability 
of the wave equation. See also \cite{BuDeLe24a,BuDeLe24b} for the extension to manifolds with boundary, \cite{Ambrosio-Crippa14} for a different proof of the propagation result on manifolds without boundary, and \cite{Burq97} for earlier work regarding $C^{2}$ metrics. For an analogous application of the results in this article, we refer to \cite{Rauch-Taylor74,Taylor15} for the argument 
that a straightfoward variation of Theorem \ref{thm:divform} and Remark \ref{rem:divreg} implies observability of the wave equation, as well as uniform energy decay for the damped wave equation.  
A minor advantage here is that, unlike in \cite{BuDeLe24}, we do not assume that the underlying manifold has a smooth structure. Instead, using reasoning as in Section \ref{subsec:Ricci}, the results apply directly to rough manifolds.

\subsection{Proof of the main result}\label{subsec:proof}

The proof follows Taylor's argument from \cite[Section 3.11]{Taylor00}, which in turn builds on the work of Bony in \cite{Bony81}. 

Firstly, for $p\in C^{r}_{*}S^{m}_{1,0}$ as in Theorem \ref{thm:mainintro} and for $0<\delta\leq 1$, one applies a symbol decomposition to write $p=p^{\sharp}_{\delta}+p^{\flat}_{\delta}$, where $p^{\sharp}_{\delta}\in S^{m}_{1,\delta}$ and $p^{\flat}_{\delta}\in C^{r}_{*}S^{m-\delta r}_{1,\delta}$. 
 Loosely speaking, this decomposition separates the frequences of 
 $p$, and the parameter $\delta$ determines the scale at which this splitting takes place. 

Now,  
one can show that an operator with a $C^{r}_{*}S^{m-\delta r}_{1,\delta}$ symbol maps $H^{s+m-\delta r}(\Rn)$ to $H^{s}(\Rn)$ for $-(1-\delta)r<s<r$, and the endpoints of the Sobolev interval can be included for the operators $p^{\flat}_{\delta}(x,D)$ if $p\in\HT^{r,\infty}(\Rn)\subseteq \HT^{r,\infty}S^{0}_{1,0}$. This implies that, if $u\in H^{s+m-\delta r}(\Rn)$ satisfies $p(x,D)u\in H^{s}(\Rn)$, then  
\begin{equation}\label{eq:propexplain}
p^{\sharp}_{\delta}(x,D)u=p(x,D)u-p^{\flat}_{\delta}(x,D)u\in H^{s}(\Rn).
\end{equation} 
At least, this works for $-(1-\delta)r<s<r$, whereas for smaller $s$ one needs to adjust the regularity of $u$, cf.~\eqref{eq:sigma}. 
We choose $\delta=1$, while \cite[Section 3.11]{Taylor00} lets $\delta<1$ be arbitrarily close to $1$. The latter leads to suboptimal results. 

Next, by \eqref{eq:propexplain}, it suffices to obtain the propagation statement with $p$ replaced by $p^{\sharp}_{\delta}$. To do so, one can rely on a positive commutator argument, albeit in a more delicate manner than usual, given that $\delta>0$. Indeed, here one uses that $p^{\sharp}_{\delta}$ is not merely an $S^{m}_{1,\delta}$ symbol, but that it has additional smoothness, allowing one to take $r$ derivatives in the spatial variables before incurring blowup. This property is used for the pseudodifferential calculus underlying the positive commutator argument. 

In fact, this pseudodifferential calculus is where a fundamental difference arises between $\delta<1$ and $\delta=1$. Whereas the calculus for operators with $S^{m}_{1,\delta}$ symbols is well behaved if $\delta<1$, even basic results break down for $\delta=1$. On the other hand, a fundamental observation by Bony is that symbols such as $p^{\sharp}_{1}$ have an additional property, involving their frequency support, that allows one to circumvent such obstacles. H\"{o}rmander in turn captured this phenomenon by observing that these symbols are contained in the class $\tilde{S}^{m}_{1,1}$ of $p\in S^{m}_{1,1}$ such that $p(x,D)^{*}=q(x,D)$ for some $q\in S^{m}_{1,1}$. Indeed, Bourdaud \cite{Bourdaud88} and H\"{o}rmander \cite{Hormander88} proved that there is an elegant and useful calculus for pseudodifferential operators with $\tilde{S}^{m}_{1,1}$ symbols, especially if the symbols have the same additional smoothness properties as $p^{\sharp}_{1}$. 

Now, it turns out that an appropriate propagation statement for precisely these kinds of $\tilde{S}^{m}_{1,1}$ symbols is already available in the literature. Indeed, while covering Bony's work on smooth nonlinear equations with rough initial data, H\"{o}rmander in \cite[Theorem 11.3.4]{Hormander97} proved a result which can also be used in the present context of linear equations with rough coefficients. In fact, \cite[Theorem 11.3.4]{Hormander97} differs from the propagation statement in \cite[Proposition 3.11.1]{Taylor00} in that the latter only applies for $\delta<1$. Moreover, \cite[Theorem 11.3.4]{Hormander97} is less abstract for $1<r<2$ than \cite[Proposition 3.11.1]{Taylor00}, allowing for the concrete statement in Theorem \ref{thm:mainintro} for such $r$. Finally, we note that, 
unlike in the case $\delta<1$ in \cite[Proposition 3.11.1]{Taylor00}, the a priori regularity of $u$ also plays a key role in \cite[Theorem 11.3.4]{Hormander97}. 

It should by now be clear that the present article is a very modest contribution to the existing literature, consisting mostly of the observation that separate results which were already available, with the possible exception of the endpoint statements in Proposition \ref{prop:pseudoLp}, can be combined to improve upon part of the current theory regarding propagation of singularities for linear equations with rough coefficients.

\subsection{Organization of this article}\label{subsec:organization}

Section \ref{sec:prelim} contains preliminaries for the rest of the article. We first collect some basics from microlocal analysis. We then introduce the relevant function spaces and use these to define rough symbol classes. Next, we state and prove the required results on paradifferential calculus.

Section \ref{sec:main} contains our main results, namely a more general version of Theorem \ref{thm:mainintro}, as well as its corollaries and auxiliary results.

Finally, in Section \ref{sec:wave} we give two applications of the main results to wave equations with rough coefficients.

\subsection{Notation}
\label{subsec:notation}

The natural numbers are $\N=\{1,2,\ldots\}$, and $\Z_{+}:=\N\cup\{0\}$. Throughout this article we fix $n\in\N$.

For $\xi\in\Rn$ we write $\lb\xi\rb:=(1+|\xi|^{2})^{1/2}$. 
We use multi-index notation, where $\partial_{\xi}
=(\partial_{\xi_{1}},\ldots,\partial_{\xi_{n}})$, 
$\partial^{\alpha}_{\xi}=\partial^{\alpha_{1}}_{\xi_{1}}\cdots\partial^{\alpha_{n}}_{\xi_{n}}$ and $\xi^{\alpha}=\xi_{1}^{\alpha_{1}}\cdots \xi^{\alpha_{n}}_{n}$
for $\xi=(\xi_{1},\ldots,\xi_{n})\in\Rn$ and $\alpha=(\alpha_{1},\ldots,\alpha_{n})\in\Z_{+}^{n}$. 

The Fourier transform of a tempered distribution $f\in\Sw'(\Rn)$ is denoted by  
$\widehat{f}$. 
If $f\in L^{1}(\Rn)$, then $\wh{f}(\xi)=\int_{\Rn}e^{-i x\cdot \xi}f(x)\ud x$ for all $\xi\in\Rn$. 
We let $\ph(D)$ be the Fourier multiplier with symbol $\ph\in\Sw'(\Rn)$.

\section{Preliminaries}\label{sec:prelim}

In this section we collect background that will be used in the rest of this article.

\subsection{Microlocal analysis}\label{subsec:microlocal}

We briefly recall the required basic microlocal analysis.

For $\Omega\subseteq \Rn$ open, $T^{*}\Omega$ is the cotangent bundle of $\Omega$, identified with $\Omega\times\Rn$, and $o:=\Omega\times\{0\}\subseteq T^{*}\Omega$ is the zero section. A subset $\Gamma\subseteq T^{*}\Omega\setminus o$ is conic if $(x,\la\xi)\in\Gamma$ for all $(x,\xi)\in\Gamma$ and $\lambda>0$.

Let $m\in\R$ and $\delta\in[0,1]$. The class $S^{m}_{1,\delta}$ consists of all $p\in C^{\infty}(\R^{2n})$ such that 
\begin{equation}\label{eq:Hormanderclass}
\sup_{(x,\xi)\in\R^{2n}}\lb \xi\rb^{-m+|\alpha|-|\beta|\delta}|\partial_{x}^{\beta}\partial_{\xi}^{\alpha}p(x,\xi)|<\infty
\end{equation}
for all $\alpha,\beta\in\Z_{+}^{n}$. 
The pseudodifferential operator $p(x,D):\Sw(\Rn)\to\Sw(\Rn)$ associated with such a symbol $p$ 
 is given by
\begin{equation}\label{eq:pseudodef}
p(x,D)f(x):=\frac{1}{(2\pi)^{n}}\int_{\Rn}e^{ix\cdot\xi}p(x,\xi)\wh{f}(\xi)\ud\xi,
\end{equation}
for $f\in\Sw(\Rn)$ and $x\in\Rn$. By adjoint action, $p(x,D)$ extends to all of $\Sw'(\Rn)$.  

For any $\delta<1$ and $p\in S^{m}_{1,\delta}$, there exists a $q\in S^{m}_{1,\delta}$ such that $p(x,D)^{*}=q(x,D)$.  
In general, this is not the case when $\delta=1$. The class $\tilde{S}^{m}_{1,1}$ consists of those $p\in S^{m}_{1,1}$ for which there exists a $q\in S^{m}_{1,1}$ such that $p(x,D)^{*}=q(x,D)$. 

Let $\Omega\subseteq \Rn$ be open and $p\in C^{1}(T^{*}\Omega\setminus o)$. Then the characteristic set of $p$ is
\begin{equation}\label{eq:Sigma}
\Sigma_{p}:=\{(x,\xi)\in T^{*}\Omega\setminus o\mid p(x,\xi)=0\}.
\end{equation}
Moreover, the Hamilton vector field $H_{p}$ associated with $p$ is given by $H_{p}(x,\xi):=(\partial_{\xi}p(x,\xi),-\partial_{x}p(x,\xi))$ for $(x,\xi)\in T^{*}\Omega\setminus o$. An integral curve of $H_{p}$ is a continuously differentiable solution $\gamma:I\to T^{*}\Omega\setminus o$ to the differential equation $\dot{\gamma}(t)=H_{p}(\gamma(t))$, 
defined on some interval $I\subseteq \R$. If $\gamma(t)\in \Sigma_{p}$ for all $t\in I$, then we call $\gamma$ a \emph{null bicharacteristic} of $p$. 

A $p\in S^{m}_{1,\delta}$ is  \emph{elliptic} on a conic set $\Gamma\subseteq\Tp\setminus o$ if there exist $c,\ka>0$ such that $|p(y,\eta)|\geq c|\eta|^{m}$ for all $(y,\eta)\in\Gamma$ with $|\eta|\geq \ka$. 
The $H^{s}(\Rn)$ \emph{wavefront set} $\WF^{s}u$ of a $u\in\Sw'(\Rn)$ is the complement in $\Tp\setminus o$ of the set of $(x,\xi)\in\Tp\setminus o$ for which there exists a $p\in S^{0}_{1,0}$, elliptic in a neighborhood of $(x,\xi)$, such that $p(x,D)u\in H^{s}(\Rn):=\lb D\rb^{-s}L^{2}(\Rn)$. This notion extends in the obvious manner to distributions on an open set $\Omega\subseteq\Rn$. 

\subsection{Function spaces}\label{subsec:spaces}

In this subsection we define the relevant function spaces.

Throughout, fix a Littlewood--Paley decomposition $(\psi_{j})_{j=0}^{\infty}\subseteq C^{\infty}_{c}(\Rn)$. That is, 
$\sum_{j=0}^{\infty}\psi_{j}(\xi)=1$ 
for all $\xi\in\Rn$, $\psi_{0}(\xi)=0$ if $|\xi|>1$, $\psi_{1}(\xi)=0$ if $|\xi|\notin [1/2,2]$, and $\psi_{j}(\xi)=\psi_{1}(2^{-j+1}\xi)$ for all $j>1$. In fact, we may assume that $\psi_{0}$ is radial and non-negative, with $\psi_{0}(\xi)=1$ for $|\xi|\leq 1/2$, and that
\begin{equation}\label{eq:LittlePaleyspecial}
\psi_{j}(\xi)=\psi_{0}(2^{-j}\xi)-\psi_{0}(2^{-j+1}\xi)
\end{equation} 
for all $j\geq1$ and $\xi\in\Rn$.

\begin{definition}\label{def:Zygmund}
For $r\in\R$, the \emph{Zygmund space} $C^{r}_{*}(\Rn)$ consists of those $f\in\Sw'(\Rn)$ such that $\psi_{j}(D)f\in L^{\infty}(\Rn)$ for all $j\geq0$, and 
\[
\|f\|_{C^{r}_{*}(\Rn)}:=\sup_{j\geq0}2^{jr}\|\psi_{j}(D)f\|_{L^{\infty}(\Rn)}<\infty.
\]
\end{definition}

We note that $C^{r}_{*}(\Rn)$ is equal to the Besov space $B^{r}_{\infty,\infty}(\Rn)$. However, the present notation is more convenient for us, and it has been used frequently in paradifferential calculus (see e.g.~\cite{Taylor91,Taylor00,Taylor23c}).

Next, let  
$\BMO(\Rn)$ be the space of functions of bounded mean oscillation (see e.g.~\cite{Stein93}). We will work with the following local version of this space.

\begin{definition}\label{def:bmo}
The space $\bmo(\Rn)$ consists of all $f\in\Sw'(\Rn)$ such that $\psi_{0}(D)f\in L^{\infty}(\Rn)$ and $(1-\psi_{0}(D))f\in\BMO(\Rn)$, endowed with the norm
\[
\|f\|_{\bmo(\Rn)}:=\|\psi_{0}(D)f\|_{L^{\infty}(\Rn)}+\|(1-\psi_{0}(D))f\|_{\BMO(\Rn)}.
\] 
Moreover, $\HT^{r,\infty}(\Rn):=\lb D\rb^{-r}\bmo(\Rn)$ for $r\in\R$.
\end{definition}

Finally, let $r=l+t>0$ for $l\in\Z_{+}$ and $t\in(0,1]$. For $r\notin\N$, we denote by $C^{r}(\Rn)$ the space of $f\in C^{l}(\Rn)$, with bounded derivatives up to order $l$, such that that $\partial_{x}^{\alpha}f$ is H\"{o}lder continuous with exponent $t$ for every $\alpha\in\Z_{+}^{n}$ with $|\alpha|=l$. Moreover, $C^{l,1}(\Rn)$ is the space of $f\in C^{l}(\Rn)$, with bounded derivatives up to order $l$, such that $\partial_{x}^{\alpha}f$ is Lipschitz for every $\alpha\in\Z_{+}^{n}$ with $|\alpha|=l$.

It is instructive to compare these spaces 
using embeddings, cf.~\cite{Triebel10}.  For example, 
\[
C^{r+\veps}_{*}(\Rn)\subsetneq \HT^{r,\infty}(\Rn)\subsetneq C^{r}_{*}(\Rn)
\]
for all $r\in\R$ and $\veps>0$. Moreover, let
 $r=l+t>0$ for $l\in\Z_{+}$ and $t\in(0,1]$. Then
\[
\HT^{r,\infty}(\Rn)\subsetneq C^{r}_{*}(\Rn)=C^{r}(\Rn)
\]
if $r\notin\N$, i.e.~if $t\in(0,1)$, and
\begin{equation}\label{eq:Zygmundemb2}
C^{l,1}(\Rn)
\subsetneq \HT^{r,\infty}(\Rn)\subsetneq C^{r}_{*}(\Rn)
\end{equation}
if $r\in\N$, i.e.~if $t=1$. 

For $X$ one of the function spaces defined so far, we let $X_{\loc}$ consist of those $f\in\Sw'(\Rn)$ such that $\psi f\in X$ for each $\psi\in C^{\infty}_{c}(\Rn)$.

\subsection{Rough symbols}\label{subsec:rough}

In this subsection we consider versions of the symbols from \eqref{eq:Hormanderclass} that have limited regularity in the spatial variable $x$, as measured in terms of the function spaces from the previous subsection.

\begin{definition}\label{def:symbolrough}
Let $r>0$, $m\in\R$, $\delta\in[0,1]$ 
and $X\in\{C^{r}_{*},\HT^{r,\infty}\}$. Then $XS^{m}_{1,\delta}$ consists of those $p:\R^{2n}\to\C$ such that the following properties hold:
\begin{enumerate}
\item\label{it:symbolrough1} $p(x,\cdot)\in C^{\infty}(\Rn)$ 
for all $x\in\Rn$, and
\[
\sup_{(x,\xi)\in\R^{2n}}\lb\xi\rb^{-m+|\alpha|}|\partial_{\xi}^{\alpha}p(x,\xi)|<\infty
\]
for each $\alpha\in\Z_{+}^{n}$; 
\item\label{it:symbolrough2} $\partial_{\xi}^{\alpha}p(\cdot,\xi)\in X(\Rn)$ for all $\xi\in\Rn$ and $\alpha\in\Z_{+}^{n}$,  
and
\[
\sup_{\xi\in\Rn}\lb\xi\rb^{-m+|\alpha|-r\delta}\|\partial_{\xi}^{\alpha}p(\cdot,\xi)\|_{X(\Rn)}<\infty.
\]
\end{enumerate} 
\end{definition}

Clearly, we can extend the definition of the pseudodifferential operator 
from \eqref{eq:pseudodef} to symbols $p$ as in Definition \ref{def:symbolrough}, at least as a map $p(x,D):\Sw(\Rn)\to\Sw'(\Rn)$. 

We will also work with symbols that have less regularity than an element from $S^{m}_{1,\delta}$ for $\delta<1$, but more than a typical element of $S^{m}_{1,1}$.

\begin{definition}\label{def:reduced}
Let $m,\mu\in\R$ and $p\in S^{m}_{1,1}$. Then $p$ has \emph{reduced order} $\mu$ if there exists an $M\in\Z_{+}$ such that $\partial_{x}^{\beta}p\in \tilde{S}^{\mu+|\beta|}_{1,1}$ for all $\beta\in\Z_{+}^{n}$ with $|\beta|\geq M$.
\end{definition}

Finally, we describe a smoothing procedure that decomposes a rough symbol as a sum of a smooth term and a rough term of lower differential order. Let $(\psi_{j})_{j=0}^{\infty}$ be the Littlewood--Paley decomposition from \eqref{eq:LittlePaleyspecial}. 
Let  $r>0$, $m\in\R$, $p\in C^{r}_{*}S^{m}_{1,0}$ and  $\delta\in(0,1]$, and for $(x,\xi)\in\R^{2n}$ set
\[
p^{\sharp}_{\delta}(x,\xi):=\sum_{k=0}^{\infty}\big(\psi_{0}(2^{-\delta k}D)p(\cdot,\xi)\big)(x)\psi_{k}(\xi)
\]
and
\begin{equation}\label{eq:flat}
p^{\flat}_{\delta}(x,\xi):=p(x,\xi)-p^{\sharp}_{\delta}(x,\xi)=\sum_{k=0}^{\infty}\big((1-\psi_{0})(2^{-\delta k}D)p(\cdot,\xi)\big)(x)\psi_{k}(\xi).
\end{equation}
We will almost exclusively deal with the case where $\delta=1$, so for simplicity of notation we write $p^{\sharp}:=p^{\sharp}_{1}$ and $p^{\flat}:=p^{\flat}_{1}$. 

This decomposition connects Definitions \ref{def:symbolrough} and \ref{def:reduced}, in the following manner.

\begin{lemma}\label{lem:smoothing}
Let $r>1$, $m\in\R$
and $p\in C^{r}_{*}S^{m}_{1,0}$. Then $p^{\flat}\in C^{r}_{*}S^{m-r}_{1,1}$, and $p^{\sharp}\in \tilde{S}^{m}_{1,1}$ is such that $\partial_{x}^{\beta}p^{\sharp}\in \tilde{S}^{m}_{1,1}$ has reduced order $m-r+1$ for all $\beta\in\Z_{+}^{n}$ with $|\beta|=1$. 
If $p$ is real-valued, then $p^{\sharp}$ is real-valued as well. 
\end{lemma}
\begin{proof}
The first statement is a special case of \cite[equation (3.27) in Chapter 1]{Taylor00}. Moreover, since $C^{r}_{*}(\Rn)\subseteq C^{1}(\Rn)$, it follows from \cite[Proposition 1.3.D]{Taylor91} and its proof that $\partial_{x}^{\beta}p^{\sharp}\in S^{m}_{1,1}$ for all $\beta\in\Z_{+}^{n}$ with $|\beta|\leq 1$, while $\partial_{x}^{\beta}p^{\sharp}\in S^{m+|\beta|-r}_{1,1}$ if $|\beta|>r$. Next, by applying \cite[Theorem 3.4.F]{Taylor91} to $\partial_{x}^{\beta}p^{\sharp}$ and then relying on \cite[Theorem 9.4.2]{Hormander97}, it follows that one may in fact replace $S^{m}_{1,1}$ and $S^{m+|\beta|-r}_{1,1}$ by $\tilde{S}^{m}_{1,1}$ and $\tilde{S}^{m+|\beta|-r}_{1,1}$, respectively. Finally, since $\psi_{0}$ is real-valued and radial, so is its inverse Fourier transform, which concludes the proof. 
\end{proof}

\begin{remark}\label{rem:secondder}
Under the conditions of Lemma \ref{lem:smoothing}, one has $\partial_{x}^{\beta}p^{\sharp}\in \tilde{S}^{m+1}_{1,1}$ for all $\beta\in\Z_{+}^{n}$ with $|\beta|=2$. This observation will be used in the proof of Theorem \ref{thm:divform}, and it follows from the same reasoning as in the proof of Lemma \ref{lem:smoothing}.
\end{remark}

\subsection{Paradifferential calculus}\label{subsec:paradifferential}

In this subsection we collect the basic results from paradifferential calculus which will be used to prove our main results.

The following mapping properties of rough pseudodifferential operators constitute one part of the proof of our main result.

\begin{proposition}\label{prop:pseudoLp}
Let $r>0$, $m\in\R$, $\delta\in[0,1]$ 
and $p\in C^{r}_{*}S^{m}_{1,\delta}$. Then 
\begin{equation}\label{eq:pseudoLp}
p(x,D):H^{s+m}(\Rn)\to H^{s}(\Rn)
\end{equation}
for all $-(1-\delta)r<s<r$. 
Moreover, if $p=q^{\flat}$ for some $q\in \HT^{r,\infty}(\Rn)$, then \eqref{eq:pseudoLp} holds for all $0\leq s\leq r$, with $m=-r$. 
\end{proposition}
\begin{proof}
The first 
statement is a special case of \cite[Theorem 2.3]{Marschall88}. 

For the second statement, 
first note that
\begin{align*}
q^{\flat}(x,D)f(x)&=\sum_{k=0}^{\infty}(1-\psi_{0}(2^{-k}D))q(x)\psi_{k}(D)f(x)\\
&=\sum_{k=0}^{\infty}\sum_{j=k+1}^{\infty}\psi_{j}(D)q(x)\psi_{k}(D)f(x)\\
&=\sum_{j=1}^{\infty}\psi_{j}(D)q(x)\sum_{k=0}^{j-1}\psi_{k}(D)f(x)\\
&=\sum_{j=1}^{\infty}\psi_{0}(2^{-j+1}D)f(x)\psi_{1}(2^{-j+1}D)q(x)
\end{align*}
for all $f\in H^{s-r}(\Rn)$ and $x\in\Rn$, by \eqref{eq:LittlePaleyspecial}. Hence \cite[Theorem 3.5.F]{Taylor91} yields the second statement for $r\in\N$, with the natural extension to $r=0$. In fact, for such $r$ there exists a $C_{1}\geq0$ independent of $q$ 
such that
\begin{equation}\label{eq:qflatbound}
\|q^{\flat}(x,D)f\|_{H^{s}(\Rn)}\leq C_{1}\|q\|_{\HT^{r,\infty}(\Rn)}\|f\|_{H^{s-r}(\Rn)},
\end{equation}
for all $f\in H^{s-r}(\Rn)$ and $0\leq s\leq r$.

To obtain the second statement for $r\in(0,\infty)\setminus \N$, we use complex interpolation. 
Let $\lfloor r\rfloor\in\Z_{+}$ be the largest integer smaller than $r$. 
For $z\in\C$ with $0\leq \Real(z)\leq 1$ and for $x\in\Rn$, set 
\[
q_{z}(x):=e^{(-z+r-\lfloor r\rfloor)^{2}}\lb D\rb^{-z+r-\lfloor r\rfloor}q(x).
\]
Then $q_{z}\in\HT^{\lfloor r\rfloor+\Real z,\infty}(\Rn)$, and $\|q_{z}\|_{\HT^{\lfloor r\rfloor+\Real z,\infty}(\Rn)}\leq C_{2}\|q\|_{\HT^{r,\infty}(\Rn)}$ 
for a $C_{2}\geq0$ independent of $z$ and $q$. Hence the endpoint cases of \eqref{eq:qflatbound} 
yield
\begin{align}
\label{eq:flatboundpart1}\|q_{z}^{\flat}(x,D)f\|_{L^{2}(\Rn)}&\leq C_{1}C_{2}\|q\|_{\HT^{r,\infty}(\Rn)}\|f\|_{H^{-(\lfloor r\rfloor+\Real(z))}(\Rn)},\\
\label{eq:flatboundpart2}\|q_{z}^{\flat}(x,D)f\|_{H^{\lfloor r\rfloor+\Real(z)}(\Rn)}&\leq C_{1}C_{2}\|q\|_{\HT^{r,\infty}(\Rn)}\|f\|_{L^{2}(\Rn)},
\end{align}
for all $z\in\C$ with $\Real z\in\{0,1\}$ and all $f\in\Sw(\Rn)$. If $f$ in addition has compact Fourier support, then it is straightforward to check that the map $z\mapsto q_{z}^{\flat}(x,D)f$ is continuous on $\{z\in\C\mid 0\leq \Real z\leq 1\}$ with values in $L^{2}(\Rn)$, and holomorphic on $\{z\in\C\mid 0< \Real z<1\}$.

Finally, note that $q_{r-\lfloor r\rfloor}=q$, and recall that the Schwartz functions with compact Fourier support are dense in $H^{\sigma}(\Rn)$ for all $\sigma\in\R$. Hence one can apply interpolation of analytic families of operators (see e.g.~\cite[Theorem 2.1.7]{Lunardi09}) to \eqref{eq:flatboundpart1} and \eqref{eq:flatboundpart2}, to obtain the required statement at the endpoints $s=0$ and $s=r$, respectively. The intermediate values of $s$ follow from the first statement.  
\end{proof}

\begin{remark}\label{rem:multiplication}
Clearly $S^{m}_{1,1}\subseteq C^{r}_{*}S^{m}_{1,1}$ for all $m\in\R$ and $r>0$. Hence, by Proposition \ref{prop:pseudoLp} and duality, any $p\in \tilde{S}^{m}_{1,1}$ satisfies $p(x,D):H^{s+m}(\Rn)\to H^{s}(\Rn)$ for all $s\in\R$. 

By using Lemma \ref{lem:smoothing} and this well-known fact, which in fact characterizes $\tilde{S}^{m}_{1,1}$ (see \cite[Theorem 9.4.2]{Hormander97}), to deal with $q^{\sharp}$ for a given $q\in \HT^{r,\infty}(\Rn)$, and by applying the second statement in 
Proposition \ref{prop:pseudoLp} to $q^{\flat}$, 
we find that multiplication by $q$ acts boundedly on $H^{s}(\Rn)$ for all $-r\leq s\leq r$. 
\end{remark}
For the next two results, recall the definition of $\hat{p}$ from \eqref{eq:phat}. The following proposition will be used to deal with the elliptic region in our main result.

\begin{proposition}\label{prop:inverse}
Let $m,\mu,s\in\R$. Let $p\in S^{m}_{1,1}$ have reduced order $\mu$, and suppose that $\hat{p}$ is a well-defined element of $C^{1}(\Tp\setminus o)$. Let $u\in H^{s+\mu}(\Rn)$. Then $\WF^{s+m}u\subseteq \Sigma_{\hat{p}}\cup \WF^{s}p(x,D)u$. 
\end{proposition}
\begin{proof}
This is a direct consequence of \cite[Theorem 9.6.7]{Hormander97}, given that $p$ is elliptic in a neighborhood of $(x,\xi)\in\Tp\setminus o$ if and only if $(x,\xi)\notin \Sigma_{\hat{p}}$.
\end{proof}

Finally, the following key proposition on propagation of singularities for $\tilde{S}^{m}_{1,1}$ symbols is essentially a special case of \cite[Theorem 11.3.4]{Hormander97}. 

\begin{proposition}\label{prop:sharpprop}
Let $m,\mu,s\in\R$. Let $p\in \tilde{S}^{m}_{1,1}$ be such that $\partial_{x}^{\beta}p\in \tilde{S}_{1,1}^{m}$ has reduced order $\mu+1$ for all $\beta\in\Z_{+}^{n}$ with $|\beta|=1$, such that $\Imag p\in \tilde{S}^{m-1}_{1,1}$, and such that $\hat{p}$ is a well-defined element of $C^{1}(\Tp\setminus o)$. 
Let $u\in H^{s+\mu}(\Rn)$. Then, for each $(x,\xi)\in \WF^{s+m-1}u\setminus \WF^{s}p(x,D)u$, there exist an interval $I\subseteq \R$ and a null bicharacteristic $\gamma:I\to \Tp\setminus o$ of $\hat{p}$ such that $(x,\xi)\in \gamma(I)\subseteq \WF^{s+m-1}u\setminus \WF^{s}p(x,D)u$ and such that $\gamma$ is maximally extended within $\Tp\setminus (o\cup \WF^{s}p(x,D)u)$. 
\end{proposition}
\begin{proof}
Set 
\[
U:=\{(y,\eta)\in \Tp\setminus o\mid H_{\hat{p}}(y,\eta)\neq 0, (y,\eta)\notin \WF^{s}p(x,D)u\}.
\]
The statement is trivial if $H_{\hat{p}}(x,\xi)=0$. If this is not the case, then $(x,\xi)\in \WF^{s+m-1}u\cap U$, and \cite[Theorem 11.3.4]{Hormander97} yields an interval $I'\subseteq \R$ and a null bicharacteristic $\gamma:I'\to \Tp\setminus o$ of $\hat{p}$, such that $(x,\xi)\in \gamma(I')\subseteq \WF^{s+m-1}u\cap U$ and such that $\gamma$ is maximally extended within the complement of $H_{\hat{p}}^{-1}(\{0\})\cup \WF^{s}p(x,D)u$. Moreover, if $I'\neq \R$ and if $\gamma(t)$ converges to an element in the null set of $H_{\hat{p}}$ as $t$ converges to a finite endpoint of $I'$, then $\gamma$ can be extended in a trivial manner within $\WF^{s+m-1}u$ after reaching this endpoint.
\end{proof}

\section{Propagation of singularities}\label{sec:main}

In this section we will prove our main results on propagation of singularities for rough symbols.

\subsection{Main result}\label{subsec:main}

With the results on paradifferential calculus from Section \ref{subsec:paradifferential} in hand, we can easily prove our main result. 

\begin{theorem}\label{thm:main}
Let $r>1$, $m\in\R$ and 
$-r<s<r$.  
Let $p\in C^{r}_{*}S^{m}_{1,0}$ be such that $\Imag p\in C^{r}_{*}S^{m-1}_{1,0}$, and such that $\hat{p}$ is a well-defined element of $C^{1}(\Tp\setminus o)$. Set
\begin{equation}\label{eq:sigmamain}
\sigma:=\begin{cases}
s-r&\text{if }0<s<r,\\
\veps-r&\text{if }-r<s\leq 0,
\end{cases}
\end{equation}
for $\veps>0$, and let $u\in H^{\sigma+m}(\Rn)$. Then
\begin{equation}\label{eq:WFinclusion}
\WF^{s+m}u\subseteq \Sigma_{\hat{p}}\cup \WF^{s}p(x,D)u.
\end{equation}
Moreover, for each $(x,\xi)\in \WF^{s+m-1}u\setminus \WF^{s}p(x,D)u$, there exist an interval $I\subseteq \R$ and a null bicharacteristic $\gamma:I\to \Tp\setminus o$ of $\hat{p}$ such that $(x,\xi)\in \gamma(I)\subseteq \WF^{s+m-1}u\setminus \WF^{s}p(x,D)u$ and such that $\gamma$ is maximally extended within $\Tp\setminus (o\cup \WF^{s}p(x,D)u)$.

Suppose additionally that 
$p(x,\xi)=\sum_{|\alpha|\leq m}a_{\alpha}(x)\xi^{\alpha}$ for all $(x,\xi)\in\R^{2n}$, where $a_{\alpha}\in \HT^{r,\infty}(\Rn)$ for each $\alpha\in\Z_{+}^{n}$ with $|\alpha|\leq m$, and $a_{\alpha}$ is real-valued if $|\alpha|=m$. Then the statements above hold for all $-r\leq s\leq r$, with $\veps=0$.
\end{theorem}
\begin{proof}
First note that $p(x,D)u\in H^{\sigma}(\Rn)\cup L^{2}(\Rn)$ is well defined, by Proposition \ref{prop:pseudoLp} and Remark \ref{rem:multiplication}.

Write $p=p^{\sharp}+p^{\flat}$, as in \eqref{eq:flat}. Then, by Lemma \ref{lem:smoothing}, $p^{\sharp}\in \tilde{S}^{m}_{1,1}$ is such that, for all $\beta\in\Z_{+}^{n}$ with $|\beta|=1$, the symbol $\partial_{x}^{\beta}p^{\sharp}\in\tilde{S}^{m}_{1,1}$ has reduced order $m-r+1$, and as such also reduced order $\mu+1$ for any $\mu\geq m-r$. 
Moreover, $\Imag p^{\sharp}=(\Imag p)^{\sharp}\in \tilde{S}^{m-1}_{1,1}$ and $p^{\flat}\in C^{r}_{*}S^{m- r}_{1,1}$. In particular, since $p^{\flat}$ has lower differential order,
\begin{equation}\label{eq:keylimit}
\lim_{\la\to\infty}\la^{-m}p^{\sharp}(x,\la\xi)=\lim_{\la\to\infty}\la^{-m}p(x,\la\xi)-\lim_{\la\to\infty}\la^{-m}p^{\flat}(x,\la\xi)=\hat{p}(x,\xi)
\end{equation}
for all $(x,\xi)\in\Tp\setminus o$. Also, 
\begin{equation}\label{eq:keyWF}
\WF^{s}p^{\sharp}(x,D)u=\WF^{s}p(x,D)u,
\end{equation}
as follows by relying on Proposition \ref{prop:pseudoLp} to see that $p^{\flat}(x,D)u\in H^{s}(\Rn)$. 

The required statements are now a consequence of Propositions \ref{prop:inverse} and \ref{prop:sharpprop}, applied to $p^{\sharp}$ 
and in combination with \eqref{eq:keylimit} and \eqref{eq:keyWF}.  
\end{proof}

\begin{remark}\label{rem:lowerorder}
One can also add rough lower-order perturbations in Theorem \ref{thm:main}, 
by analyzing the proof. For example, let $r>1$ and suppose that $p=p_{0}+p_{1}$ for $p_{0}\in C^{r}_{*}S^{m}_{1,0}$ and $p_{1}\in C^{r-1}_{*}S^{m-1}_{1,0}$ such that $\Imag p_{0}\in C^{r}_{*}S^{m-1}_{1,0}$, and such that $\hat{p_{0}}$ is a well-defined element of $C^{1}(\Tp\setminus o)$. Then $p^{\sharp}$ still satisfies the conditions of Propositions \ref{prop:inverse} and \ref{prop:sharpprop} with $\mu=m-r$, and $p^{\flat}(x,D):H^{\sigma+m}(\Rn)\to H^{s}(\Rn)$ is bounded for all $-(r-1)<s<r-1$. Hence the conclusion of Theorem \ref{thm:main} still holds, albeit for a smaller range of $s$. Of course, the endpoints $s=-(r-1)$ and $s=r-1$ of this Sobolev interval can again be included under conditions as in Theorem \ref{thm:main}, although these are now only required of $p_{1}$. 
If one wants to allow $\veps=0$ in \eqref{eq:sigmamain} for $s\leq 0$, then additional assumptions should also be made on $p_{0}$. 
\end{remark}

\begin{remark}\label{rem:othersymbols}
Theorem \ref{thm:main} is stated for symbols in $C^{r}_{*}S^{m}_{1,0}$, but one can obtain an analogous result for $C^{r}_{*}S^{m}_{1,\theta}$ symbols with $\theta\in(0,1)$. To do so, it suffices to modify Lemma \ref{lem:smoothing}. Namely, for $p\in C^{r}_{*}S^{m}_{1,\theta}$, 
the term $p^{\sharp}$ has the same properties as in Lemma \ref{lem:smoothing}, while $p^{\flat}\in C^{r}_{*}S^{m-(1-\theta)r}_{1,1}$. Hence one can repeat the proof of Theorem \ref{thm:main} and obtain an extension of that result, with $\sigma$ replaced by
\[
\sigma_{\theta}:=\begin{cases}
s-(1-\theta)r&\text{if }0<s<r,\\
\veps-(1-\theta)r&\text{if }-r<s\leq 0.
\end{cases}
\]
\end{remark}

\begin{remark}\label{rem:C11}
If $\hat{p}\in C^{1,1}(\Tp\setminus o)$, then integral curves of $H_{\hat{p}}$ are unique.
In fact, by Osgood's Theorem (see \cite[Proposition 1.A.2]{Taylor23a}), slightly less than $C^{1,1}$ regularity guarantees uniqueness. For example, $\hat{p}\in C^{2}_{*,\loc}(\Tp\setminus o)$ suffices (see \cite[Section 3.11]{Taylor00}). 
\end{remark}

The statement on propagation of singularities in Theorem \ref{thm:main} is equivalent to the following statement on propagation of regularity.

\begin{corollary}\label{cor:propreg1} 
Let $r>1$, $m\in\R$ and 
$-r< s<r$. 
Let $p\in C^{r}_{*}S^{m}_{1,0}$ be 
such that $\Imag p\in C^{r}_{*}S^{m-1}_{1,0}$, and such that $\hat{p}$ is a well-defined element of $C^{1}(\Tp\setminus o)$.  
Let $\sigma$ be as in \eqref{eq:sigmamain} for $\veps>0$, and let $u\in H^{\sigma+m}(\Rn)$. Let $\Gamma_{1},\Gamma_{2}\subseteq \Tp\setminus (o\cup \WF^{s}p(x,D)u)$ be conic subsets, and suppose that every null bicharacteristic of $\hat{p}$ which passes through $\Gamma_{1}$, and which is maximally extended within $\Tp\setminus (o\cup \WF^{s}p(x,D)u)$, also passes through $\Gamma_{2}$. Then, if $\Gamma_{2}\cap \WF^{s+m-1}u=\emptyset$, also $\Gamma_{1}\cap \WF^{s+m-1}u=\emptyset$.

Suppose additionally that $p(x,\xi)=\sum_{|\alpha|\leq m}a_{\alpha}(x)\xi^{\alpha}$ for all $(x,\xi)\in\R^{2n}$, where $a_{\alpha}\in \HT^{r,\infty}(\Rn)$ for each $\alpha\in\Z_{+}^{n}$ with $|\alpha|\leq m$, and $a_{\alpha}$ is real-valued if $|\alpha|=m$. Then the statement above holds for all $-r\leq s\leq r$, with $\veps=0$.
\end{corollary}

With a bit of extra work, Corollary \ref{cor:propreg2} can be quantified, in the following manner. 

\begin{corollary}\label{cor:propreg2}
Let $r>1$, $m\in\R$ and 
$-r< s<r$. 
Let $p\in C^{r}_{*}S^{m}_{1,0}$ be 
such that $\Imag p\in C^{r}_{*}S^{m-1}_{1,0}$, and such that $\hat{p}$ is a well-defined element of $C^{1}(\Tp\setminus o)$.  
Let $\sigma$ be as in \eqref{eq:sigmamain} for $\veps>0$. Let $\Gamma_{1},\Gamma_{2},\Gamma_{2}\subseteq \Tp\setminus o$ and $q_{1},q_{2},q_{3}\in S^{0}_{1,0}$ be such that the following properties hold:
\begin{itemize}
\item $\Gamma_{1},\Gamma_{2},\Gamma_{3}$ are conic, and $\Gamma_{1},\Gamma_{2}\subseteq \Gamma_{3}$;
\item every null bicharacteristic of $\hat{p}$ which passes through $\Gamma_{1}$, and which is maximally extended within $\Gamma_{3}$, also passes through $\Gamma_{2}$;
\item $\supp(q_{1})\subseteq \Gamma_{1}$, $q_{2}$ is elliptic on $\Gamma_{2}$, and $q_{3}$ is elliptic on $\Gamma_{3}$.
\end{itemize}
Write $Q_{j}:=q_{j}(x,D)$ for $1\leq j\leq 3$. Then there exists a $C\geq0$ such that $Q_{1}u\in H^{s+m-1}(\Rn)$ and
\[
\|Q_{1}u\|_{H^{s+m-1}(\Rn)}\leq C\big(\|Q_{2}u\|_{H^{s+m-1}(\Rn)}+\|Q_{3}p(x,D)u\|_{H^{s}(\Rn)}+\|u\|_{H^{\sigma+m}(\Rn)}\big),
\]
for all $u\in H^{\sigma+m}(\Rn)$ such that $Q_{2}u\in H^{s+m-1}(\Rn)$ and $Q_{3}p(x,D)u\in H^{s}(\Rn)$. 

Suppose additionally that $p(x,\xi)=\sum_{|\alpha|\leq m}a_{\alpha}(x)\xi^{\alpha}$ for all $(x,\xi)\in\R^{2n}$, where $a_{\alpha}\in \HT^{r,\infty}(\Rn)$ for each $\alpha\in\Z_{+}^{n}$ with $|\alpha|\leq m$, and $a_{\alpha}$ is real-valued if $|\alpha|=m$. Then the statement above holds for all $-r\leq s\leq r$, with $\veps=0$.
\end{corollary}
\begin{proof}
First note that $Q_{2}$ and $Q_{3}p(x,D)$ are closed unbounded operators from $H^{\sigma+m}(\Rn)$ to $H^{s+m-1}(\Rn)$ and from $H^{\sigma+m}(\Rn)$ to $H^{s}(\Rn)$, respectively. For $Q_{2}$ this follows from the fact that $Q_{2}^{*}:\Sw(\Rn)\to \Sw(\Rn)$ continuously, while for $Q_{3}p(x,D)$ one can use that $Q_{3}^{*}:H^{-\sigma}(\Rn)\to H^{-\sigma}(\Rn)$ and that $p(x,D):H^{\sigma+m}(\Rn)\to H^{\sigma}(\Rn)$, by Proposition \ref{prop:pseudoLp}.

Now set 
\[
X:=\{u\in H^{\sigma+m}(\Rn)\mid Q_{2}u\in H^{s+m-1}(\Rn), Q_{3}p(x,D)u\in H^{s}(\Rn)\}
\]
and, for $u\in X$,
\[
\|u\|_{X}:=\|Q_{2}u\|_{H^{s+m-1}(\Rn)}+\|Q_{3}p(x,D)u\|_{H^{s}(\Rn)}+\|u\|_{H^{\sigma+m}(\Rn)}.
\]
By the previous observations, $X$ is a Banach space. Moreover, 
for every $u\in X$, by assumption and by standard pseudodifferential calculus, $\Gamma_{2}\cap \WF^{s+m-1}u=\emptyset$ and $\Gamma_{3}\cap \WF^{s}p(x,D)u=\emptyset$. Hence, by Corollary \ref{cor:propreg1}, $\Gamma_{1}\cap \WF^{s+m-1}u=\emptyset$ and thus $Q_{1}u\in H^{s+m-1}(\Rn)$. Again, $Q_{1}:X\to H^{s+m-1}(\Rn)$ is a closed operator, and the closed graph theorem therefore concludes the proof. 
\end{proof}

\begin{remark}\label{rem:constants}
Typically, one obtains a qualitative statement as in Corollary \ref{cor:propreg1} as a consequence of a quantitative statement as in Corollary \ref{cor:propreg2}. However, to prove Theorem \ref{thm:main}, and thus Corollary \ref{cor:propreg1}, we relied on the qualitative statement in Proposition \ref{prop:sharpprop}. The latter is in turn obtained in \cite{Hormander97} as a consequence of \cite[Lemma 11.3]{Hormander97} and an additional argument. Whereas \cite[Lemma 11.3]{Hormander97} is proved using the positive commutator method, and as such can be quantified in the same manner as usual, the additional argument which yields Proposition \ref{prop:sharpprop} is more qualitative. Instead of expecting the reader to quantify the proof of Proposition \ref{prop:sharpprop} in \cite{Hormander97}, we have simply included an abstract proof of the quantitative bounds.
\end{remark}

\subsection{Auxiliary results}\label{subsec:auxiliary}

Theorem \ref{thm:main} is formulated for operators defined over the full space $\Rn$, whereas for applications to e.g.~manifolds one is often interested in equations on domains $\Omega\subseteq \Rn$. We will not delve into subtleties regarding rough pseudodifferential operators on domains or manifolds, but we do include a corollary of Theorem \ref{thm:main} concerning differential operators on domains. 

\begin{corollary}\label{cor:domains}
Let $r>1$, $m\in\Z_{+}$ and 
$-r<s<r$.  
Let $\Omega\subseteq \Rn$ be open. For each $\alpha\in\Z_{+}^{n}$ with $|\alpha|\leq m$, let $a_{\alpha}\in C^{r}_{*,\loc}(\Omega)$ be such that $a_{\alpha}$ is real-valued if $|\alpha|=m$. Set $p(x,\xi):=\sum_{|\alpha|\leq m}a_{\alpha}(x)\xi^{\alpha}$ and $p_{m}(x,\xi):=\sum_{|\alpha|=m}a_{\alpha}(x)\xi^{\alpha}$, for $(x,\xi)\in T^{*}\Omega$. Let $\sigma$ be as in \eqref{eq:sigmamain} for $\veps>0$, and let $u\in H^{\sigma+m}_{\loc}(\Omega)$. Then
\begin{equation}\label{eq:WFinclusiondomain}
\WF^{s+m}u\subseteq \Sigma_{p_{m}}\cup \WF^{s}p(x,D)u.
\end{equation}
Moreover, for each $(x,\xi)\in \WF^{s+m-1}u\setminus \WF^{s}p(x,D)u$, there exist an interval $I\subseteq \R$ and a null bicharacteristic $\gamma:I\to T^{*}\Omega\setminus o$ of $p_{m}$ such that $(x,\xi)\in \gamma(I)\subseteq \WF^{s+m-1}u\setminus \WF^{s}p(x,D)u$ and such that $\gamma$ is maximally extended within $T^{*}\Omega\setminus (o\cup \WF^{s}p(x,D)u)$.

Suppose additionally that $(a_{\alpha})_{|\alpha|\leq m}\subseteq \HT^{r,\infty}_{\loc}(\Rn)$. Then the statements above hold for all $-r\leq s\leq r$, with $\veps=0$.
\end{corollary}
\begin{proof}
First note that $p(x,D)u\in H^{\sigma}_{\loc}(\Omega)\cup L^{2}_{\loc}(\Rn)$ is well defined. Indeed, given $x\in \Omega$ and a real-valued $\psi\in C^{\infty}_{c}(\Omega)$ satisfying $\psi\equiv 1$ in a neighborhood of $x$, let $\ph\in C^{\infty}_{c}(\Omega)$ be real-valued and such that $\ph\equiv 1$ on $\supp(\psi)$. Then
\begin{equation}\label{eq:localidentity}
p(y,D)u(y)=\psi(y)p(y,D)u(y)=\psi(y)p(y,D)(\ph u)(y) 
\end{equation}
for all $y\in \Omega$ near $x$. Hence Proposition \ref{prop:pseudoLp} and Remark \ref{rem:multiplication}, applied to $\psi p\in C^{r}_{*}S^{m}_{1,0}$ and $\ph u\in H^{\sigma+m}(\Rn)$, imply that $\psi p(x,D)u\in H^{\sigma}(\Rn)\cup L^{2}_{\loc}(\Rn)$.

Next, observe that $\hat{p}=p_{m}$, 
and similarly for $\psi p$. 
Hence, for \eqref{eq:WFinclusiondomain} 
one can again apply 
 \eqref{eq:localidentity}, this time in combination with Theorem \ref{thm:main}.  
Indeed, if $(x,\xi)\notin \Sigma_{p_{m}}\cup \WF^{s}p(x,D)u$ then $(x,\xi)\notin \Sigma_{\psi p_{m}}\cup \WF^{s}\psi p(x,D)(\ph u)$, and \eqref{eq:WFinclusion} implies that $(x,\xi)\in \WF^{s+m}\ph u$. Since $\ph\equiv 1$ near $x$, one has $(x,\xi)\in \WF^{s+m}u$.

The second statement follows in a similar manner. By Theorem \ref{thm:main}, for any $(x,\xi)\in \WF^{s+m-1}u\setminus \WF^{s}p(x,D)u$, there exists an open interval $I_{0}\subseteq\R$ and a 
null bicharacteristic $\gamma_{0}:I\to \Tp\setminus 0$ of $\psi p_{m}$ such that $(x,\xi)\in \gamma_{0}(I_{0})\subseteq \WF^{s+m-1}\ph u\setminus \WF^{s}\psi p(x,D)(\ph u)$.  
Upon possibly shrinking $I_{0}$, one then  
in fact has $\gamma_{0}(I_{0})\subseteq \WF^{s+m-1}u\setminus \WF^{s}p(x,D)u$, and $\gamma_{0}$ is a null bicharacteristic of $p_{m}$. 
Since $\WF^{s+m-1}u\setminus \WF^{s}p(x,D)u$ is closed in $T^{*}\Omega\setminus (o\cup \WF^{s}p(x,D)u)$, this suffices.
\end{proof}

\begin{remark}\label{rem:lowerorderdomains}
As in Remark \ref{rem:lowerorder},  one can include lower-order terms of lower regularity in Corollary \ref{cor:domains}, 
and in some cases these will only affect the size of the Sobolev interval for $s$. For example, with notation as in Corollary \ref{cor:domains}, suppose that 
$(a_{\alpha})_{|\alpha|=m}\subseteq C^{r}_{*,\loc}(\Omega)$  
and  
$(a_{\alpha})_{|\alpha|<m}\subseteq C^{r-1}_{*,\loc}(\Omega)$. Then the conclusion of Corollary \ref{cor:domains} holds for all $-(r-1)<s<r-1$. 
Moreover, if $(a_{\alpha})_{|\alpha|<m}\subseteq \HT^{r-1,\infty}_{\loc}(\Omega)$, then one may include $s=-(r-1)$ and $s=r-1$, and one may let $\veps=0$ in \eqref{eq:sigmamain} if also $(a_{\alpha})_{|\alpha|=m}\subseteq \HT^{r,\infty}_{\loc}(\Omega)$. 
\end{remark}

Finally, we include a result for second-order operators in divergence form. Such operators can be dealt with using Corollary \ref{cor:domains} and Remark \ref{rem:lowerorderdomains}, but doing so would lead to a smaller Sobolev interval than we will now obtain.

\begin{theorem}\label{thm:divform}
Let $r>1$ and 
$-r-1<s<r-1$. 
Let $\Omega\subseteq \Rn$ be open, let $(a_{ij})_{i,j=1}^{n}\subseteq C^{r}_{*,\loc}(\Omega)$ be real-valued, and set $P:=-\sum_{i,j=1}^{n}\partial_{x_{i}}a_{ij}\partial_{x_{j}}$ and $p(x,\xi):=\sum_{i,j=1}^{n}a_{ij}(x)\xi_{i}\xi_{j}$ for $(x,\xi)\in T^{*}\Omega$. Set
\begin{equation}\label{eq:tau}
\tau:=\begin{cases}
s-r&\text{if }-1<s<r-1,\\
-1+\veps-r&\text{if }-r-1<s\leq -1,
\end{cases}
\end{equation}
for $\veps>0$, and let $u\in H_{\loc}^{\tau+2}(\Omega)$. Then
\[
\WF^{s+2}u\subseteq \Sigma_{p}\cup \WF^{s}Pu.
\]
Moreover, for each $(x,\xi)\in \WF^{s+1}u\setminus \WF^{s}Pu$, there exist an interval $I\subseteq \R$ and a null bicharacteristic $\gamma:I\to T^{*}\Omega\setminus o$ of $p$ such that $(x,\xi)\in \gamma(I)\subseteq \WF^{s+m-1}u\setminus \WF^{s}Pu$ and such that $\gamma$ is maximally extended within $T^{*}\Omega\setminus (o\cup \WF^{s}Pu)$.

Suppose additionally that $(a_{ij})_{i,j=1}^{n}\subseteq \HT^{r,\infty}(\Rn)$. Then the statements above
hold for all $-r-1\leq s\leq r-1$, with $\veps=0$. 
\end{theorem}
\begin{proof}
We first consider the case where $\Omega=\Rn$, $(a_{ij})_{i,j=1}^{n}\subseteq C^{r}_{*}(\Rn)$ and $u\in H^{\tau+2}(\Rn)$. The boundedness properties of multiplication by $a_{ij}$, contained in Proposition \ref{prop:pseudoLp} and Remark \ref{rem:multiplication}, imply that $Pu\in H^{\tau}(\Rn)\cup H^{-1}(\Rn)$ is well defined. Apply the symbol decomposition to each $a_{ij}$, to write
\begin{align*}
P&=p^{\sharp}(x,D)-\sum_{i,j=1}^{n}(\partial_{x_{i}}a_{ij}^{\sharp})(x,D)\partial_{x_{j}}-\sum_{i,j=1}^{n}\partial_{x_{i}}a_{ij}^{\flat}(x,D)\partial_{x_{j}}\\
&=q(x,D)-\sum_{i,j=1}^{n}\partial_{x_{i}}a_{ij}^{\flat}(x,D)\partial_{x_{j}},
\end{align*}
where $q(x,\xi):=p^{\sharp}(x,\xi)-i\sum_{i,j=1}^{n}\partial_{x_{i}}a_{ij}^{\sharp}(x,\xi)\xi_{j}$ for $(x,\xi)\in\Rn\times\Rn$. 

By Lemma \ref{lem:smoothing} and Remark \ref{rem:secondder}, for all $\beta\in\Z_{+}^{n}$ with $|\beta|=1$, the symbol $\partial_{x}^{\beta}q\in \tilde{S}_{1,1}^{2}$ has reduced order $3-r$. 
Moreover, $\Imag q\in \tilde{S}^{1}_{1,1}$ and $\hat{q}=\hat{p}=p\in C^{1}(\Tp\setminus o)$, where we used the embedding $C^{r}_{*}(\Rn)\subseteq C^{1}(\Rn)$. On the other hand, $\sum_{i,j=1}^{n}\partial_{x_{i}}a_{ij}^{\flat}(x,D)\partial_{x_{j}}:H^{\tau+2}(\Rn)\to H^{s}(\Rn)$ is bounded, by Proposition \ref{prop:pseudoLp}. Hence $\WF^{s}q(x,D)u=\WF^{s}Pu$, and one can again apply Propositions \ref{prop:inverse} and \ref{prop:sharpprop} 
to conclude the proof under these global assumptions.

For general $\Omega$, $(a_{ij})_{i,j=1}^{n}\subseteq C^{r}_{*,\loc}(\Omega)$ and $u\in H_{\loc}^{\tau+2}(\Omega)$, one can reason as in the proof of Corollary \ref{cor:domains} and rely on what we have already shown.
\end{proof}

\begin{remark}\label{rem:divreg}
Of course, in  
Corollary \ref{cor:domains} and Theorem \ref{thm:divform}, 
as in Corollaries \ref{cor:propreg1} and \ref{cor:propreg2}, regularity of $u$ propagates along bicharacteristics. 
And, as in Remark \ref{rem:C11}, if the principal coefficients of $p$ are contained in 
$C^{2}_{*,\loc}(\Rn)$, then integral curves of $H_{\hat{p}}$ 
are unique and one recovers the classical notion of propagation of singularities.
\end{remark}

\section{Wave equations}\label{sec:wave}

In this section we use the results from the previous section to recover and improve some results in the literature concerning wave equations with rough coefficients.

\subsection{Operators in standard form}\label{subsec:standwave}

Let $r>1$, 
let $(a_{ij})_{i,j=1}^{n},(b_{j})_{j=1}^{n}\subseteq C^{r}_{*,\loc}(\R^{n+1})$ be real-valued, and let $(c_{j})_{j=0}^{n}\subseteq C^{r-1}_{*,\loc}(\R^{n+1})$. Set
\[
p(t,x,\tau,\xi):=\tau^{2}-\sum_{i,j=1}^{n}a_{ij}(t,x)\xi_{i}\xi_{j}-\sum_{j=1}^{n}b_{j}(t,x) \xi_{j}\tau+c_{0}(t,x)\tau+\sum_{j=1}^{n}c_{j}\xi_{j}
\]
for $(t,x,\tau,\xi)\in T^{*}\R^{n+1}$. Then Corollary \ref{cor:domains} and Remark \ref{rem:lowerorderdomains} yield propagation of $H^{s+1}_{\loc}(\R^{n+1})$ regularity of $u\in H^{\sigma+2}_{\loc}(\R^{n+1})$, for all $-(r-1)<s<r-1$, with $\sigma$ as in \eqref{eq:sigmamain} for $\veps>0$. 
Moreover, if $(c_{j})_{j=0}^{n}\subseteq\HT^{r-1,\infty}_{\loc}(\R^{n+1})$, then one may include $s=-(r-1)$ and $s=r-1$, and if additionally $(a_{ij})_{i,j=1}^{n},(b_{j})_{j=1}^{n}\subseteq \HT^{r,\infty}_{\loc}(\R^{n+1})$, then one may also let $\veps=0$. 

The same statement follows from \cite[Section 3.11]{Taylor00} for $r=2$, but with $\sigma$ replaced by $\sigma+\delta$ for any $\delta>0$. Moreover, for $r=2$ the endpoint case $\delta=0$ is dealt with in \cite{Smith14}, albeit for $s\in\{0,1\}$ and 
under the assumption that the equation is uniformly hyperbolic  in $t$, that the $a_{ij}$ and $b_{j}$ have second derivatives in $L^{1}_{t}L^{\infty}_{x}$, and that the $c_{j}$ have first derivatives in $L^{1}_{t}L^{\infty}_{x}$. Note that the latter regularity assumptions complement those in this article, in the sense that in \cite{Smith14} less regularity is required in the time variable but more in the spatial variable. 

In fact, in \cite{Smith14} a limiting procedure is used to extend the main results to coefficients that are piecewise continuous in time. It is then shown that the a priori regularity assumption $u\in H^{s}_{\loc}(\R^{n+1})$ on the solution  
is sharp for such piecewise continuous coefficients. On the other hand, there is no distinguished time variable in the setting of Corollary \ref{cor:domains}. Moreover, the sharpness example from \cite{Smith14} does not directly apply here, given that Corollary \ref{cor:domains} requires in each variable more regularity than piecewise continuity.

\subsection{Bounded Ricci tensor}\label{subsec:Ricci}

Let $M$ be a compact $n$-dimensional $C^{1}$ manifold without boundary, with a continuous metric tensor $g$, i.e.~an inner product on tangent vectors. 
Then the Laplace--Beltrami operator $\Delta_{g}:H^{1}(M)\to H^{-1}(M)$ is well defined, as is the notion of a harmonic function on $M$. 

Suppose that there exist $1<r<2$ and 
$c,K_{0},K_{1}>0$ with the following property. For every $x\in M$, there exists an open neighborhood $U_{x}\subseteq M$ of $x$ and a $C^{1}$ diffeomorphism $\Phi_{x}:B\to U_{x}$ such that $\Phi_{x}(0)=x$, and such that the pull-back of $g=(g_{ij})_{i,j=1}^{n}$ to $B$ satisfies $K_{0}^{-1}I\leq (g_{ij}(y))_{i,j=1}^{n}\leq K_{0}I$ for all $y\in B$, as well as $g_{ij}(0)=\delta_{ij}$ and
\begin{equation}\label{eq:metricreg}
g_{ij}\in C^{r-1}(B)\cap H^{1}(B),
\end{equation}
with $\|g_{ij}\|_{C^{r-1}(B)\cap H^{1}(B)}\leq K_{1}$, for all $1\leq i,j\leq n$. Here $B\subseteq\Rn$ is the open ball around $0$ of radius $c$, $I$ is the identity matrix and $\delta_{ij}$ is the Kronecker delta. Then, as noted in e.g.~\cite[Section 3.9]{Taylor00}, $M$ is in fact a manifold of class $C^{r}\cap H^{2}$, and the original coordinate maps $\Phi_{x}$ are $C^{r}\cap H^{2}$ diffeomorphisms, for every $x\in M$. Moreover, without loss of generality, one may assume that each $\Phi_{x}^{-1}$ is harmonic.

Now, the connection $1$-form is $\Gamma:=\sum_{j=1}^{n}\Gamma_{j}d x_{j}$ in local coordinates on $B$, where $\Gamma_{j}$ is the matrix with entries 
\[
{\Gamma^{a}}_{bj}:=\frac{1}{2}\sum_{m=1}^{n}g^{am}(\partial_{j}g_{bm}+\partial_{b}g_{jm}-\partial_{m}g_{bj}),
\]
for $1\leq a,b,j\leq n$. Note that $\Gamma$ is well defined and square integrable on $B$, by \eqref{eq:metricreg}. Also, the Riemannian curvature tensor is 
\[
R:=d\Gamma+\Gamma\wedge \Gamma,
\]
and it has components ${R^{a}}_{bjk}\in H^{-1}(B)+ L^{1}(B)$ for $1\leq a,b,j,k\leq n$, again by \eqref{eq:metricreg}. Finally, the Ricci curvature tensor is the $2$-form with components
\[
\Ric_{bk}:=\sum_{j=1}^{n}{R^{j}}_{bjk}\in H^{-1}(B)+L^{1}(B),
\]
for $1\leq b,k\leq n$. Note that if the Riemannian tensor is bounded, then so is the Ricci tensor. Moreover, if $(g_{ij})_{i,j=1}^{n}\subseteq C^{1,1}(B)$, then the Riemannian tensor is bounded.

On the other hand, using that we are working in harmonic coordinates, one can show as in \cite[Section 3.10]{Taylor00} that if 
$(\Ric_{bk})_{b,k=1}^{n}\subseteq L^{\infty}(B)$, 
or more generally if $(\Ric_{bk})_{b,k=1}^{n}\subseteq \HT^{0,\infty}(B)$, then $(g_{ij})_{i,j=1}^{n}\subseteq \HT^{2,\infty}(B)$. And, by \cite[Proposition 3.1.13]{Taylor00}, the latter in turn implies that the harmonic coordinate atlas is of class $W^{3,q}$ for any $q<\infty$. As such, $M$ is then in fact a manifold of class $W^{3,q}$.

By arguing as in \cite[Section 2]{Chen-Smith19}, the Sobolev spaces $H^{s}(M)$ can now be defined in local coordinates for $-3\leq s\leq 3$, and for $|s|\leq 2$ the resulting spaces coincide with those defined using the spectral decomposition of $\Delta_{g}$ on $L^{2}(M)$. In turn, the spaces $H^{s}_{\loc}(\R\times M)$ can then also be defined for $-3\leq s\leq 3$, in local coordinates.

Finally, recall that the wave equation $\partial_{t}^{2}u=\Delta_{g}u$ can be written as
\[
\partial_{t}\sqrt{|g|}\partial_{t}u-\sum_{i,j=1}^{n}\partial_{x_{i}}\big(\sqrt{|g|}g^{ij}\partial_{x_{j}}u\big)=0,
\]
in local coordinates, where $|g|:=\det (g_{ij})_{i,j=1}^{n}$ and $(g^{ij})_{i,j=1}^{n}$ is the inverse of $(g_{ij})_{i,j=1}^{n}$. Hence 
 Theorem \ref{thm:divform} and Remark \ref{rem:divreg}
yield, with $\tau$ as in \eqref{eq:tau}, propagation of $H^{s+1}_{\loc}(\R\times M)$ 
regularity of solutions $u\in H^{\tau+2}_{\loc}(\R\times M)$ to $\partial_{t}^{2}u=\Delta_{g}u$, for all 
$-3\leq s\leq 1$. 
This improves upon the corresponding result in 
\cite[Section 3.11]{Taylor00}, which contains the same statement but with $\tau$ replaced by $\tau+\delta$ for any $\delta>0$, and without the endpoints of the Sobolev interval for $s$.

\section*{Acknowledgments}\label{sec:acknowledge}

 The author would like to thank Andrew Hassell for helpful suggestions.

\bibliographystyle{plain}
\bibliography{Bibliography}

\begin{thebibliography}{10}

\bibitem{Ambrosio-Crippa14}
Luigi Ambrosio and Gianluca Crippa.
\newblock Continuity equations and {ODE} flows with non-smooth velocity.
\newblock {\em Proc. Roy. Soc. Edinburgh Sect. A}, 144(6):1191--1244, 2014.

\bibitem{Bony81}
Jean-Michel Bony.
\newblock Calcul symbolique et propagation des singularit\'{e}s pour les
  \'{e}quations aux d\'{e}riv\'{e}es partielles non lin\'{e}aires.
\newblock {\em Ann. Sci. \'{E}cole Norm. Sup. (4)}, 14(2):209--246, 1981.

\bibitem{Bourdaud88}
G\'erard Bourdaud.
\newblock Une alg\`ebre maximale d'op\'erateurs pseudo-diff\'erentiels.
\newblock {\em Comm. Partial Differential Equations}, 13(9):1059--1083, 1988.

\bibitem{Burq97}
Nicolas Burq.
\newblock Contr\^{o}labilit\'{e} exacte des ondes dans des ouverts peu
  r\'{e}guliers.
\newblock {\em Asymptot. Anal.}, 14(2):157--191, 1997.

\bibitem{BuDeLe24a}
Nicolas Burq, Belhassen Dehman, and J\'{e}r\^{o}me Le~Rousseau.
\newblock Measure and continuous vector field at a boundary {I}: propagation
  equations and wave observability.
\newblock Preprint available at \url{https://arxiv.org/abs/2407.02255}, 2024.

\bibitem{BuDeLe24b}
Nicolas Burq, Belhassen Dehman, and J\'{e}r\^{o}me Le~Rousseau.
\newblock Measure and continuous vector field at a boundary {II}: geodesics and
  support propagation.
\newblock Preprint available at \url{https://arxiv.org/abs/2407.02259}, 2024.

\bibitem{BuDeLe24}
Nicolas Burq, Belhassen Dehman, and J\'er\^ome Le~Rousseau.
\newblock Measure propagation along a {${\mathcal C}^0$}-vector field and wave
  controllability on a rough compact manifold.
\newblock {\em Anal. PDE}, 17(8):2683--2717, 2024.

\bibitem{Chen-Smith19}
Yuanlong Chen and Hart~F. Smith.
\newblock Dispersive estimates for the wave equation on {R}iemannian manifolds
  of bounded curvature.
\newblock {\em Pure Appl. Anal.}, 1(1):101--148, 2019.

\bibitem{deHoUhVa15}
Maarten de~Hoop, Gunther Uhlmann, and Andr\'as Vasy.
\newblock Diffraction from conormal singularities.
\newblock {\em Ann. Sci. \'{E}c. Norm. Sup\'{e}r. (4)}, 48(2):351--408, 2015.

\bibitem{Duistermaat-Hormander72}
Johannes~J. Duistermaat and L.~H{\"o}rmander.
\newblock Fourier integral operators. {II}.
\newblock {\em Acta Math.}, 128(3-4):183--269, 1972.

\bibitem{Hormander71b}
Lars H{\"o}rmander.
\newblock Fourier integral operators. {I}.
\newblock {\em Acta Math.}, 127(1-2):79--183, 1971.

\bibitem{Hormander71a}
Lars H\"{o}rmander.
\newblock On the existence and the regularity of solutions of linear
  pseudo-differential equations.
\newblock {\em Enseign. Math. (2)}, 17:99--163, 1971.

\bibitem{Hormander88}
Lars H\"ormander.
\newblock Pseudo-differential operators of type {$1,1$}.
\newblock {\em Comm. Partial Differential Equations}, 13(9):1085--1111, 1988.

\bibitem{Hormander97}
Lars H\"ormander.
\newblock {\em Lectures on nonlinear hyperbolic differential equations},
  volume~26 of {\em Math\'ematiques \& Applications (Berlin) [Mathematics \&
  Applications]}.
\newblock Springer-Verlag, Berlin, 1997.

\bibitem{Lunardi09}
Alessandra Lunardi.
\newblock {\em Interpolation {T}heory}.
\newblock Appunti. Scuola Normale Superiore di Pisa (Nuova Serie). [Lecture
  Notes. Scuola Normale Superiore di Pisa (New Series)]. Edizioni della
  Normale, Pisa, second edition, 2009.

\bibitem{Marschall88}
J\"{u}rgen Marschall.
\newblock Pseudodifferential operators with coefficients in {S}obolev spaces.
\newblock {\em Trans. Amer. Math. Soc.}, 307(1):335--361, 1988.

\bibitem{Melrose-Sjostrand78}
Richard~B. Melrose and Johannes Sj\"ostrand.
\newblock Singularities of boundary value problems. {I}.
\newblock {\em Comm. Pure Appl. Math.}, 31(5):593--617, 1978.

\bibitem{Melrose-Sjostrand82}
Richard~B. Melrose and Johannes Sj\"ostrand.
\newblock Singularities of boundary value problems. {II}.
\newblock {\em Comm. Pure Appl. Math.}, 35(2):129--168, 1982.

\bibitem{Rauch-Taylor74}
Jeffrey Rauch and Michael Taylor.
\newblock Exponential decay of solutions to hyperbolic equations in bounded
  domains.
\newblock {\em Indiana Univ. Math. J.}, 24:79--86, 1974.

\bibitem{Smith14}
Hart~F. Smith.
\newblock Propagation of singularities for rough metrics.
\newblock {\em Anal. PDE}, 7(5):1137--1178, 2014.

\bibitem{Stein93}
Elias~M. Stein.
\newblock {\em Harmonic analysis: real-variable methods, orthogonality, and
  oscillatory integrals}, volume~43 of {\em Princeton Mathematical Series}.
\newblock Princeton University Press, Princeton, NJ, 1993.
\newblock With the assistance of Timothy S. Murphy, Monographs in Harmonic
  Analysis, III.

\bibitem{Taylor91}
Michael~E. Taylor.
\newblock {\em Pseudodifferential {O}perators and {N}onlinear {PDE}}, volume
  100 of {\em Progress in Mathematics}.
\newblock Birkh\"{a}user Boston, Inc., Boston, MA, 1991.

\bibitem{Taylor00}
Michael~E. Taylor.
\newblock {\em Tools for {PDE}}, volume~81 of {\em Mathematical Surveys and
  Monographs}.
\newblock American Mathematical Society, Providence, RI, 2000.
\newblock Pseudodifferential operators, paradifferential operators, and layer
  potentials.

\bibitem{Taylor15}
Michael~E. Taylor.
\newblock Wave decay on manifolds with bounded {R}icci tensor, and related
  estimates.
\newblock {\em J. Geom. Anal.}, 25(2):1018--1044, 2015.

\bibitem{Taylor23a}
Michael~E. Taylor.
\newblock {\em Partial differential equations {I}. {B}asic theory}, volume 115
  of {\em Applied Mathematical Sciences}.
\newblock Springer, Cham, third edition, [2023] \copyright 2023.

\bibitem{Taylor23c}
Michael~E. Taylor.
\newblock {\em Partial differential equations {III}. {N}onlinear equations},
  volume 117 of {\em Applied Mathematical Sciences}.
\newblock Springer, Cham, third edition, [2023] \copyright 2023.

\bibitem{Triebel10}
Hans Triebel.
\newblock {\em Theory of {F}unction spaces}.
\newblock Modern Birkh\"auser Classics. Birkh\"auser/Springer Basel AG, Basel,
  2010.
\newblock Reprint of 1983 edition.

\bibitem{Vasy08}
Andr\'as Vasy.
\newblock Propagation of singularities for the wave equation on manifolds with
  corners.
\newblock {\em Ann. of Math. (2)}, 168(3):749--812, 2008.

\end{thebibliography}

\end{document}